\documentclass[11pt]{amsart}
\usepackage{amsmath, color}
\usepackage{amsfonts}
\usepackage{latexsym}
\usepackage{ifpdf}
\usepackage{graphicx,amssymb,lineno}
\usepackage{amssymb}
\usepackage{mathrsfs}
\usepackage{cite}
\usepackage{extarrows}
\usepackage{enumerate}
\usepackage{ulem}
\usepackage{tikz}
\usepackage{float}
\usepackage{makecell}
\usepackage{scalefnt}
\usetikzlibrary{calc}
\allowdisplaybreaks[4]
\numberwithin{equation}{section}

\newtheorem{theorem}{{\bf Theorem}}[section]

\newtheorem{corollary}[theorem]{{\bf Corollary}}
\newtheorem{remark}{{\bf Remark}}[section]

\newcommand{\Sp}{\mathrm{Sp}}
\newcommand{\SO}{\mathrm{SO}}

\date{}
\begin{document}
\title{\Large Universal symplectic/orthogonal functions and general branching rules}
\author{Zhihong Jin}
\address{School of Science, Huzhou University, Huzhou, Zhejiang 313000, China}
\email{317598173@qq.com}
\author{Naihuan Jing$^{\dagger}$}
\address{Department of Mathematics, North Carolina State University, Raleigh, NC 27695, USA}
\email{jing@ncsu.edu}
\author{Zhijun Li}
\address{School of Science, Huzhou University, Huzhou, Zhejiang 313000, China}
\email{zhijun1010@163.com}
\author{Danxia Wang}
\address{School of Science, Huzhou University, Huzhou, Zhejiang 313000, China}
\email{dxwangmath@126.com}
\thanks{{\scriptsize
\hskip -0.6 true cm MSC (2010): Primary: 05E05; Secondary: 17B37.
\newline Keywords: Universal symplectic/orthgonal functions, Vertex operator realizations, General branching rules, $CB$-interpolating Schur functions.
\newline $^\dag$ Corresponding author: jing@ncsu.edu
\newline Supported by Simons Foundation (grant no. MP-TSM-00002518), NSFC (grant nos. 12171303, 12101231) and NSF of Huzhou (grant no. 2022YZ47).
}}
\maketitle
\begin{abstract} In this paper, we first introduce a family of universal symplectic functions $sp_\lambda(\mathbf{x}^{\pm};\mathbf{z})$  that include
symplectic Schur functions $sp_\lambda(\mathbf{x}^{\pm})$, odd symplectic characters $sp_\lambda(\mathbf{x}^{\pm};z)$, universal symplectic characters $sp_\lambda(\mathbf{z})$ and  intermediate symplectic characters as subfamilies. We then realize the universal symplectic functions
by vertex operators, which naturally lead to their skew versions, and show that $sp_\lambda(\mathbf{x}^{\pm};\mathbf{z})$
obey the general branching rules. This also gives the Gelfand-Tsetlin representations of odd symplectic characters and a transition formula between odd symplectic characters and symplectic Schur functions.

Secondly we introduce a family of universal orthogonal functions $o_\lambda(\mathbf{x}^{\pm};\mathbf{z})$ and their skew versions in a similar manner, and we provide their vertex operator realizations and
obtain transition formulas and the branching rule. The universal orthogonal functions
$o_\lambda(\mathbf{x}^{\pm};\mathbf{z})$ generalize orthogonal Schur functions $o_\lambda(\mathbf{x}^{\pm})$, odd orthogonal Schur functions $so_\lambda(\mathbf{x}^{\pm})$, universal orthogonal characters $o_\lambda(\mathbf{z})$ as well as  intermediate orthogonal characters.

Thirdly, we give vertex operator realizations for the $CB$-interpolating Schur functions $s^{CB}_\lambda(x;\beta)$ introduced by Bisi and Zygouras (Adv. Math., 2022) and
the $DB$-interpolating Schur functions $s^{DB}_\lambda(x;\beta)$ interpolating between characters of type $D$ and $B$. As an application, we show $s^{CB}_\lambda(x;\beta)$ are equal to the orthosymplectic Schur polynomials $spo_\lambda(x/\beta)$, thus give a short proof of the generalization of the Brent-Krattenthaler-Warnaar identity obtained by Kumari (arXiv:2401.01723).
\end{abstract}

\section{Introduction}
The symplectic Schur functions and (odd/even) orthogonal Schur functions form one of the important chapters in the classic monograph of Weyl \cite{Wey1946} who introduced them as the characters of irreducible representations of the classical complex groups $\Sp_{2n}$ and $\SO_{n}$ 
respectively. Proctor\cite{Pr1988} introduced an odd symplectic group $\Sp_{2n+1}$ by considering a degenerate symplectic form on an odd-dimensional space
and  derived the characters of the indecomposable representations of $\Sp_{2n+1}$
as the first example of odd symplectic characters. Universal symplectic/orthogonal characters of $\Sp_{2n}$ and $\SO_{2n}$ in the ring of symmetric functions were defined by Koike and Terada\cite{KT1987} and many identities between them and Schur functions have been derived therein.
Proctor\cite{Pr1988} studied intermediate symplectic characters
as the characters of ``trace-free'' representations of intermediate symplectic group $\Sp_{2k,n-k}$. Krattenthaler\cite{Kr1995} gave determinantal expressions for intermediate symplectic characters and also defined intermediate orthogonal characters. These latter works have called for a similar treatment of the
{\it odd orthogonal Schur functions} as a family of symmetric functions like Schur functions.

In this paper, we would like to put all these symmetric functions in a broad umbrella and introduce universal symplectic functions $sp_\lambda(\mathbf{x}^{\pm};\mathbf{z})$ and universal orthogonal functions $o_\lambda(\mathbf{x}^{\pm};\mathbf{z})$ (invariants of $(S_n\ltimes (\mathbb{Z}_2)^n)\times S_m$) which generalize  aforementioned symplectic/orthogonal-type functions so that they can be better studied by a uniformed method.

In the previous work \cite{JLW2022}, three of us have used the vertex operator realizations to construct skew symplectic Schur functions and skew orthogonal Schur functions and derived their favorable properties such as the Jacobi-Trudi formulas, Cauchy-type identities and Gelfand-Tsetlin representations etc.
Actually, the vertex operator realizations also show that these skew-type functions satisfy the general branching rules \eqref{e:br1} and \eqref{e:br2}, which have
resolved the question whether the skew versions share the same property as the skew Schur functions.
We will extend the method of \cite{JLW2022} to give vertex operator realizations for (skew) universal symplectic/orthogonal functions. One will see that
the approach of the vertex operator realizations naturally provides a unified framework to get the general branching rules for $sp_\lambda(\mathbf{x}^{\pm};\mathbf{z})$ and $o_\lambda(\mathbf{x}^{\pm};\mathbf{z})$.

With the vertex operator realizations of symmetric functions, it is natural to see that
universal symplectic/orthogonal functions generalize various important families of symmetric functions
such as odd symplectic characters, odd orthogonal Schur functions, universal symplectic/orthogonal characters and intermediate symplectic/orthogonal characters etc.
A transition formula between odd symplectic characters and symplectic Schur functions is also provided. The following table displays  the  situation of
aforementioned symplectic/orthogonal-type functions.

\begin{tabular}{|c|c|c|}
\hline
\makecell{$\mathbf{x}^{\pm}=(x^{\pm}_1,\dots,x^{\pm}_n)$,\\$\mathbf{z}=(z_1,\dots,z_m)$,\\$\lambda=(\lambda_1,\dots,\lambda_{n+m})$} & \makecell{universal symplectic functions\\ $sp_\lambda(\mathbf{x}^{\pm};\mathbf{z})$} & \makecell{universal orthogonal functions\\ $o_\lambda(\mathbf{x}^{\pm};\mathbf{z})$}\\ \hline
$m=0$ & symplectic Schur functions $sp_\lambda(\mathbf{x}^{\pm})$ & orthogonal Schur functions $o_\lambda(\mathbf{x}^{\pm})$\\ \hline
$m=1$& \makecell{{\it odd symplectic characters}\\ $sp_\lambda(\mathbf{x}^{\pm};z)$} & \makecell{(under $l(\lambda)=n$ and $z=1$)\\odd orthogonal Schur functions\\$so_\lambda(\mathbf{x}^{\pm})$} \\ \hline
$n=0$ & universal symplectic chars. $sp_\lambda(\mathbf{z})$ & universal orthogonal chars. $o_\lambda(\mathbf{z})$\\ \hline
$\lambda=(\lambda_1,\dots,\lambda_{n})$ & intermediate symplectic characters & intermediate orthogonal characters\\ \hline
\end{tabular}

\medskip

In Macdonald's theory of Hall-Littlewood and Macdonald symmetric functions \cite{Mac1995}, the skew versions have played a pivotal role in deriving combinatorial formulas for classical families of symmetric functions. Likewise, we also derive Gelfand-Tsetlin presentations for skew odd symplectic characters. It is shown that despite $\Sp_{2n+1}$ is not counted among the classical groups, its characters are also given by generating functions of Gelfand-Tsetlin patterns.

In \cite{BZ2022}, Bisi and Zygouras introduced two families of generalized Schur functions $s^{CB}_\lambda(x;\beta)$ and $s^{DB}_\lambda(x;\alpha)$
that intercalate between characters of types $C$ and $B$ and types $D$ and $B$ respectively,
We provide vertex operator realization for the $CB$-interpolating Schur functions $s^{CB}_\lambda(x;\beta)$ and a new family of $DB$-interpolating Schur functions $s^{DB}_\lambda(x;\beta)$. For $x=(x_1,\dots,x_n)$ and $y=(y_1,\dots,y_m)$, Bisi and Zygouras derived an important identity to show how a $CB$-interpolating
Schur functions of rectangular
shape can be expressed as a bounded Cauchy sum for $CB$-interpolating
Schur functions of the same
type \cite[(1.10)]{BZ2022}
\begin{align}\label{e:CB6}
s^{CB}_{u^{m+n}}(x,y;\beta)=\sum_{\lambda_1\leq u}s^{CB}_{(u^{n-m},\lambda)}(x;\beta)s^{CB}_{\lambda}(y;\beta).
\end{align}
Using vertex operator realization for $s^{CB}_\lambda(x;\beta)$ ($=spo_\mu(x/\beta)$), we will show that \eqref{e:CB6} actually agrees with the Brent-Krattenthaler-Warnaar-type identity obtained by Kumari \cite[Thm 4.7]{Kum2024}
\begin{equation}
\label{BKW}
     \sum_{\lambda} \beta^u spo_{(u^{n-m},\lambda)}(x/\beta) spo_{\lambda}(y/\beta) \\
     =
    \sum_{j=1}^{m+n+1} \beta^{u+m+n+1-j} sp_{\big(u^{j-1},(u-1)^{m+n+1-j}\big)}(\mathbf{x}^{\pm},\mathbf{y}^{\pm}),
\end{equation}
where the orthosymplectic Schur polynomials $spo_\mu(x/\beta)$ are the characters of irreducible representations of
orthosymplectic Lie superalgebras $spo(2n,1)$. As a by-product, this provides a quick and easy proof of \eqref{BKW}.

This paper is organized as follows. In Sections 2, we give vertex operator realizations for universal symplectic functions and their skew-type, and obtain general branching rules for universal symplectic functions. We also show general branching rules and some combinatorial properties for the special cases.
 In Section 3, we give vertex operator realizations and obtain some properties for universal symplectic functions and the extremal cases. In Section 4, we provide the vertex operator realizations for $CB$-interpolating Schur functions $s^{CB}_\lambda(x;\beta)$ and $DB$-interpolating Schur functions $s^{DB}_\lambda(x;\beta)$ and then derive their transition formulas and the Jacobi-Trudi identities. We also give a new proof of the Brent-Krattenthaler-Warnaar-type identity recently obtained by Kumari.

\section{Universal symplectic functions}
\subsection{Symplectic vertex operators}In this section, we recall necessary definitions and facts
about symplectic vertex operators according to the conventions of \cite{JLW2022,JN2015}. Throughout the paper, we work over an algebraically
closed field $\mathbf{k}$ of characteristic 0 and write $[A,B]=AB-BA$.

A {\it general partition} $\lambda=(\lambda_1,\dots,\lambda_n)$ is a finite sequence of decreasing {\it nonnegative integers}, and the number of nonzero parts is called the length, denoted by $l(\lambda)$, and $|\lambda|=\lambda_1+\dots+\lambda_n$ is the weight of $\lambda$\footnote{Note that a partition is a special case of general partition.}. We find it necessary to distinguish between two such sequences which differ only by a
string of zeros at the end. 
A general partition $\mu$ is contained in another general general partition
$\lambda$ interlaces another general partition $\nu$, denoted as $\nu\prec\lambda$, if $\lambda_{i}\ge\nu_{i}\geq\lambda_{i+1}$ for all $i$.

Let $\mathcal{H}$ be the Heisenberg algebra generated by $\{a_n|n\neq 0\}$ with the central element $c=1$ subject to the commutation relations\cite{FK1980}
\begin{align}\label{e:he1}
[a_m,a_n]=m\delta_{m,-n}c,~~~~[a_n,c]=0.
\end{align}
The Fock space $\mathcal{M}$ (resp. $\mathcal{M}^*$) is generated by the vacuum vector $|0\rangle$ (resp. dual vacuum vector $\langle0|$) and
subject to
\begin{align*}
a_n|0\rangle=\langle0|a_{-n},~~n>0.
\end{align*}

Define a gradation of $\mathcal M$ by $deg(a_{n})=n$ ($n>0$), and let $\mathcal M^*_n$ be the degree $n$ subspace.
Let $\widetilde{\mathcal{M}^*}$ be the completion of $\mathcal{M}^*$. Here the topology is by the sequence of subspaces generated by $\mathcal M^*_n$.

Define the vertex operators\cite{JN2015,Ba1996}
\begin{align}\label{e:sp40}
&Y(z)=\exp\left(\sum^\infty_{n=1}\frac{a_{-n}}{n}z^n\right)\exp\left(-\sum^\infty_{n=1}\frac{a_n}{n}(z^{-n}+z^n)\right)=\sum_{n\in \mathbb{Z}}Y_nz^{-n},\\
&Y^*(z)=(1-z^2)\exp\left(-\sum^\infty_{n=1}\frac{a_{-n}}{n}z^n\right)\exp\left(\sum^\infty_{n=1}\frac{a_n}{n}(z^{-n}+z^n)\right)=\sum_{n\in \mathbb{Z}}Y^*_nz^n.
\end{align}
It is easy to check that\cite{Ba1996,JN2015,JLW2022}
\begin{equation}
\begin{aligned}\label{e:com1}
&Y_iY_j+Y_{j+1}Y_{i-1}=0,\qquad Y^*_iY^*_j+Y^*_{j-1}Y^*_{i+1}=0,\qquad Y_iY^*_j+Y^*_{j+1}Y_{i+1}=\delta_{i,j},\\
&Y_n|0\rangle=Y^*_{-n}|0\rangle=0, ~~\text{for}~ n>0,~~Y_0|0\rangle=Y^*_0|0\rangle=|0\rangle,\\
&\langle 0|Y^*_n=-\langle 0|Y^*_{-n+2},~~ \langle 0|Y_n=\langle 0|Y_{-n},
\end{aligned}
\end{equation}
where the key tool is the Baker-Campbell-Hausdorff (BCH) formula
\begin{align*}
e^Ae^B=e^Be^Ae^{[A, B]},~~\mbox{if $[A,B]$ commutes with $A$ and $B$.}
\end{align*}

For general partition $\lambda=(\lambda_1,\dots,\lambda_l)$, define the {\it partition elements}
\begin{align*}
|\lambda^{sp}\rangle=Y_{-\lambda_1}Y_{-\lambda_2}\cdots Y_{-\lambda_l}|0\rangle, \qquad\langle \lambda^{sp}|=\langle 0|Y^*_{-\lambda_l}\cdots Y^*_{-\lambda_1}.
\end{align*}
We remark that $\langle \lambda^{sp}|$ belongs to $\widetilde{\mathcal{M}^*}$ other than $\mathcal{M}^*$ and that $\langle 0|Y^*_0\neq \langle 0|$. According to the commutation relations \eqref{e:com1}, for any general partition $\mu=(\mu_1,\dots,\mu_l)$ and permutation $\sigma\in S_l$ we have that
\begin{align}\label{e:sp9}
&\varepsilon(\sigma)\langle 0|\left(^{Y^*_{-\mu_{\sigma(l)}+\sigma(l)-l}}_{-Y^*_{\mu_{\sigma(l)}-\sigma(l)+l+2}}\right)\cdots \left(^{Y^*_{-\mu_{\sigma(i)}+\sigma(i)-i}}_{-Y^*_{\mu_{\sigma(i)}-\sigma(i)+2(l+1)-i}}\right)\cdots \left(^{Y^*_{-\mu_{\sigma(1)}+\sigma(1)-1}}_{-Y^*_{\mu_{\sigma(1)}-\sigma(1)+2l+1}}\right)=\langle\mu^{sp}|,\\
\label{e:sp27}&\varepsilon(\sigma)Y_{-\mu_{\sigma(1)}+\sigma(1)-1}\dots Y_{\mu_{\sigma(l)}-\sigma(l)+l}|0\rangle=|\mu^{sp}\rangle,
\end{align}
where the notation $(^a_b)$ means either $a$ or $b$. More generally, for any integer sequence $\mathbf{n}=(n_1,\dots,n_l)$,
$\langle \mathbf{n}^{sp}|=\langle 0|Y^*_{-n_l}\cdots Y^*_{-n_1}$ equals to a partition element up to sign or $0$.

 For general partitions $\mu=(\mu_1,\mu_2,\dots,\mu_l)$ and $\lambda=(\lambda_1,\lambda_2,\dots,\lambda_l)$ with possible zero parts in the tails, it follows from \cite[Theorem 2.1]{JLW2022} that\footnote{For completeness, a proof is provided in the Appendix.}
\begin{align}
\label{e:sp14}\langle \mu^{sp}|\lambda^{sp}\rangle=\delta_{\lambda\mu}.
\end{align}
By \eqref{e:sp14}, for general partition $\alpha=(\alpha_1,\dots,\alpha_n)$, we have
\begin{align}\label{e:sp15}
\langle \alpha^{sp}|\sum_{\eta=(\eta_1,\dots,\eta_n)}|\eta^{sp}\rangle\langle\eta^{sp}|=\langle \alpha^{sp}|,~~\sum_{\eta=(\eta_1,\dots,\eta_n)}|\eta^{sp}\rangle\langle\eta^{sp}||\alpha^{sp}\rangle=|\alpha^{sp}\rangle,
\end{align}
where the sum is over all general partitions $\eta=(\eta_1,\dots,\eta_n)$. Thus $\sum_{\eta=(\eta_1,\dots,\eta_n)}|\eta^{sp}\rangle\langle\eta^{sp}|$ can be viewed as an identity operator in the subspace spanned by $|\alpha^{sp}\rangle$ (or $\langle\alpha^{sp}|$). 

Introduce the following half vertex operators
\begin{align*}
 \Gamma_+(w)=\exp\left(\sum^\infty_{n=1}\frac{a_n}{n}w^n\right),~~ \Gamma_-(w)=\exp\left(\sum^\infty_{n=1}\frac{a_{-n}}{n}w^n\right).
\end{align*}
For $\mathbf{x}^{\pm}=(x^{\pm}_1,\dots,x_n^{\pm})$, $\mathbf{z}=(z_1,\dots,z_m)$, define
\begin{align}\label{e:gamma1}
&\Gamma_+(\mathbf{x}^{\pm})=\prod^n_{i=1}\Gamma_+(x_i)\Gamma_+(x^{-1}_i),\\ \label{e:gamma2}
&\Gamma_+(\mathbf{x}^{\pm};\mathbf{z})=\prod^m_{j=1}\Gamma_+(z_j)\prod^n_{i=1}\Gamma_+(x_i)\Gamma_+(x^{-1}_i).
\end{align}

We introduce the generalized homogeneous and elementary symmetric functions by the generating functions
\begin{align*}
&\prod^n_{i=1}\frac{1}{(1-x_iw)(1-x^{-1}_iw)}\prod^m_{j=1}\frac{1}{(1-z_jw)}=\sum_{i\in \mathbb{Z}}h_i(\mathbf{x}^{\pm};\mathbf{z})w^i,\\
&\prod^n_{i=1}\frac{1}{(1-x_iw)(1-x^{-1}_iw)}\prod^m_{j=1}\frac{1}{(1-z_jw)(1-z^{-1}_jw)}=\sum_{i\in \mathbb{Z}}h_i(\mathbf{x}^{\pm};\mathbf{z}^{\pm})w^i,\\
&\prod^m_{i=1}(1-z^{-1}_iw)=\sum^m_{i=0}e_i(-\mathbf{z}^{-1})w^i=\sum^m_{i=0}e_iw^i.
\end{align*}
It is easy to see that $h_i(\mathbf{x}^{\pm};\mathbf{z})=0$ for $i<0$ and $h_0(\mathbf{x}^{\pm};\mathbf{z})=1$. One then has the generalized Newton identity:
\begin{align}\label{e:sp8}
&h_n(\mathbf{x}^{\pm};\mathbf{z})=\sum^m_{i=0}e_ih_{n-i}(\mathbf{x}^{\pm};\mathbf{z}^{\pm}),\\
\label{e:sp50}&h_n(\mathbf{x}^{\pm};\mathbf{z})=\sum^\infty_{i=0}h_i(\mathbf{z})h_{n-i}(\mathbf{x}^{\pm})
\end{align}
with
\begin{align*}
\prod^n_{i=1}\frac{1}{(1-x_iw)(1-x^{-1}_iw)}=\sum_{i\in \mathbb{Z}}h_i(\mathbf{x}^{\pm})w^i, ~~~~~~~~\prod^m_{j=1}\frac{1}{(1-z_jw)}=\sum_{i\in \mathbb{Z}}h_i(\mathbf{z})w^i.
\end{align*}
%
\subsection{Vertex operator realization of (skew) symplectic Schur functions}We only recall some basic results of (skew) symplectic Schur functions from \cite{JLW2022}.
\begin{theorem}\label{th6} \rm{\cite[Prop. 3.2]{JLW2022}} For any general partition $\nu=(\nu_1,\dots,\nu_{n})$, one has
\begin{align}\label{e:sp26}
\langle0|\Gamma_+(\mathbf{x}^{\pm})|\nu^{sp}\rangle~=~sp_\nu(\mathbf{x}^{\pm})~=~&\frac{\det(x^{\lambda_j+(n-j+1)}_i-x^{-\lambda_j-(n-j+1)}_i)^n_{i,j=1}}{\det(x^{n-j+1}_i-x^{-(n-j+1)}_i)^n_{i,j=1}}\\
~=~&\det\big{(}h_{\lambda_i-i+j}(\mathbf{x}^{\pm})+\delta_{j>1}h_{\lambda_i-i-j+2}(\mathbf{x}^{\pm})\big{)}^{n}_{ i,j=1},
\end{align}
where $sp_\nu(\mathbf{x}^{\pm})$ is the symplectic Schur function associated to $\nu$ \cite{FK1997,Wey1946}.
\end{theorem}
\begin{theorem}\label{th7}\rm{\cite[Prop. 3.2]{JLW2022}}
For general partitions $\beta=(\beta_1,\dots,\beta_l)$ and $\alpha=(\alpha_1,\dots,\alpha_{l+n})$, one has
\begin{align}\label{e:sp1}
\langle\beta^{sp}|\Gamma_+(\mathbf{x}^{\pm})|\alpha^{sp}\rangle~=~&\det\big{(}h_{\alpha_i-\beta_j-i+j}(\mathbf{x}^{\pm})+\delta_{j>l+1}h_{\alpha_i-i-j+2l+2}(\mathbf{x}^{\pm})\big{)}^{l+n}_{ i,j=1}\\
~=~&sp_{\alpha/\beta}(\mathbf{x}^{\pm}),
\end{align}
which is nonzero unless $\beta\subset\alpha$ and $sp_{\alpha/\beta}(\mathbf{x}^{\pm})$ is the skew symplectic Schur function attached to $\alpha/\beta$ \cite{JLW2022,AFHS2023}.
\end{theorem}

Let $\alpha_\sigma=(\alpha_{\sigma(1)}-\sigma(1)+1,\dots,\alpha_{\sigma(l+n)}-\sigma(l+n)+l+n)\in \mathbb{Z}^{l+n}$ for $\sigma\in S_{l+n}$
(i.e. the dot action of the symmetric group). Using
determinantal properties and \eqref{e:sp27}, one can easily check the following permutation property for the skew symplectic Schur functions \cite{JLW2022}
\begin{align}\label{e:sp28}
\langle\beta^{sp}|\Gamma_+(\{\mathbf{x}^{\pm}\})|\alpha_\sigma^{sp}\rangle=sp_{\alpha_\sigma/\beta}(\mathbf{x}^{\pm})=\epsilon(\sigma)sp_{\alpha/\beta}(\mathbf{x}^{\pm}).
\end{align}

In \cite{JLW2022}, three of us have obtained the Gelfand-Tsetlin representation for skew symplectic Schur functions via the vertex operator method
and we have showed therein that the skew version obeys the branching rule
\begin{align*}
sp_\nu(x^\pm_1,\dots,x^\pm_n)=\sum_{\gamma\prec\eta\prec\nu}sp_\gamma(x^\pm_1,\dots,x^\pm_{n-1})x^{2|\eta|-|\gamma|-|\nu|}_n.
\end{align*}
and more generally they satisfy the branching rule:
\begin{align}\label{e:br1}
sp_\nu(x^\pm_1,\dots,x^\pm_n)=\sum_{\gamma=(\gamma_1,\dots,\gamma_{n-k})\subset\nu}sp_\gamma(x^\pm_1,\dots,x^\pm_{n-k})sp_{\nu/\gamma}(x^\pm_{n-k+1},\dots,x^\pm_n).
\end{align}
\subsection{Vertex operator realization of (skew) universal symplectic functions}In \cite{Ok2021}, Okada defined universal symplectic functions (invariants of $(S_n\ltimes (\mathbb{Z}_2)^n)\times S_m$) for partition $\lambda$ with $l(\lambda)\leq n$.
 \begin{align}\label{e:uni}
 sp_\lambda(\mathbf{x}^{\pm};\mathbf{z})=\det\big{(}h_{\lambda_i-i+j}(\mathbf{x}^{\pm};\mathbf{z})+\delta_{j>1}h_{\lambda_i-i-j+2}(\mathbf{x}^{\pm};\mathbf{z})\big{)}^{n+m}_{ i,j=1},
 \end{align}
 where $\delta_{j>1}$ equals 0 expect for $j>1$ it is 1.

 We now generalize Okada's definition and consider the universal symplectic functions $sp_\lambda(\mathbf{x}^{\pm};\mathbf{z})$ for partitions $\lambda$
 with $l(\lambda)\leq n+m$, where $\mathbf{x}$ and $\mathbf{z}$ are $n$-dimensional and $m$-dimensional vectors respectively.
\begin{theorem}\label{th1}
For any general partition $\lambda=(\lambda_1,\dots,\lambda_{n+m})$, one has that
\begin{align}\label{e:sp1}
\langle0|\Gamma_+(\mathbf{x}^{\pm};\mathbf{z})|\lambda^{sp}\rangle=sp_\lambda(\mathbf{x}^{\pm};\mathbf{z})=\det\big{(}h_{\lambda_i-i+j}(\mathbf{x}^{\pm};\mathbf{z})+\delta_{j>1}h_{\lambda_i-i-j+2}(\mathbf{x}^{\pm};\mathbf{z})\big{)}^{n+m}_{ i,j=1}.
\end{align}
\end{theorem}
\begin{proof} Writing
$\exp(\mathbf{z}^{-1})=\exp\left(-\sum^\infty_{k=1}\frac{a_k}{k}(z_1^{-k}+\cdots+z^{-k}_m)\right)$, and recalling the vertex operator $Y(w)$ \eqref{e:sp40}, we compute that
\begin{align*}
\exp(\mathbf{z}^{-1})Y(w)=\prod^m_{i=1}(1-\frac{w}{z_i})Y(w)\exp(\mathbf{z}^{-1}).
\end{align*}
It then follows by taking coefficients that
\begin{align*}
\exp(\mathbf{z}^{-1})Y_i=\sum^m_{s=0}e_sY_{i+s}.
\end{align*}
Subsequently 
\begin{align}
\exp(\mathbf{z}^{-1})|\lambda^{sp}\rangle=(\sum^m_{s_1=0}e_{s_1}Y_{-\lambda_1+s_1})\cdots (\sum^m_{s_{n+m}=0}e_{s_{n+m}}Y_{-\lambda_{n+m}+s_{n+m}})|0\rangle.
\end{align}

Invoking the half vertex operators \eqref{e:gamma1}-\eqref{e:gamma2}, we have that
\begin{align*}
&\langle 0|\Gamma_+(\mathbf{x}^{\pm};\mathbf{z})|\lambda^{sp}\rangle\\
&~~=\langle 0|\Gamma_+(\mathbf{x}^{\pm};\mathbf{z}^{\pm})\exp(\mathbf{z}^{-1})|\lambda^{sp}\rangle\\
&~~=\langle 0|\Gamma_+(\mathbf{x}^{\pm};\mathbf{z}^{\pm})\sum_{\eta=(\eta_1,\dots,\eta_{n+m})}|\eta^{sp}\rangle\langle\eta^{sp}|\exp(\mathbf{z}^{-1})|\lambda^{sp}\rangle\\
&~~=\sum_{\eta=(\eta_1,\dots,\eta_{n+m})}sp_{\eta}(\mathbf{x}^{\pm};\mathbf{z}^{\pm})\langle\eta^{sp}|(\sum^m_{s_1=0}e_{s_1}Y_{-\lambda_1+s_1})\cdots (\sum^m_{s_{n+m}=0}e_{s_{n+m}}Y_{-\lambda_{n+m}+s_{n+m}})|0\rangle\\
&~~=\sum^{m}_{s_1,\dots,s_{n+m}=0}e_{s_1}\cdots e_{s_{n+m}}sp_{\lambda_1-s_1,\dots,\lambda_{n+m}-s_{n+m}}(\mathbf{x}^{\pm};\mathbf{z}^{\pm})\\
&~~=\det(M_{ij})^{n+m}_{i,j=1},
\end{align*}
where the second and the fourth equations have used \eqref{e:sp15} and \eqref{e:sp28} respectively, and
\begin{equation}\label{e:sp9}
M_{ij}=\left\{\begin{aligned}
&\sum^m_{s_i=0}e_{s_i}h_{\lambda_i-s_i-\mu_j-i+j}(\mathbf{x}^{\pm};\mathbf{z}^{\pm})~~&~~j=1,\\
&\sum^m_{s_i=0}e_{s_i}h_{\lambda_i-s_i-i+j}(\mathbf{x}^{\pm};\mathbf{z}^{\pm})+\sum^m_{s_i=0}e_{s_i}h_{\lambda_i-s_i-i-j+2}(\mathbf{x}^{\pm};\mathbf{z}^{\pm})~&~~2\leq j\leq n+m.
\end{aligned}
\right.
\end{equation}
Applying the generalized Newton identity \eqref{e:sp8}, we then finish the proof.
\end{proof}

\begin{theorem}\label{th2}
For general partitions $\beta=(\beta_1,\dots,\beta_l)$ and $\alpha=(\alpha_1,\dots,\alpha_{l+n})$, one has
\begin{align}\label{e:sp1}
\langle\beta^{sp}|\Gamma_+(\mathbf{x}^{\pm};\mathbf{z})|\alpha^{sp}\rangle=&\det\big{(}h_{\alpha_i-\beta_j-i+j}(\mathbf{x}^{\pm};\mathbf{z})+\delta_{j>l+1}h_{\alpha_i-i-j+2l+2}(\mathbf{x}^{\pm};\mathbf{z})\big{)}^{l+n}_{ i,j=1}\\
=&sp_{\alpha/\beta}(\mathbf{x}^{\pm};\mathbf{z}),
\end{align}
which is zero unless $\beta\subset\alpha$ and will be called 
the skew universal symplectic function associated to $\alpha/\beta$.
\end{theorem}
\begin{proof} Let $\Gamma_+(\mathbf{z})=\prod^m_{j=1}\Gamma_+(z_j)$, then
\begin{align*}
\Gamma_+(\mathbf{z})Y(w)=\prod^m_{i=1}\frac{1}{1-z_iw}Y(w)\Gamma_+(\mathbf{z}).
\end{align*}
In terms of components,
\begin{align*}
\Gamma_+(\mathbf{z})Y_i=\sum^\infty_{s=0}h_s(\mathbf{z})Y_{i+s}.
\end{align*}
Subsequently by \eqref{e:sp27}
\begin{align}
\notag\Gamma_+(\mathbf{z})|\alpha^{sp}\rangle&=(\sum^\infty_{s_1=0}h_{s_1}(\mathbf{z})Y_{-\alpha_1+s_1})\cdots (\sum^\infty_{s_{l+n}=0}h_{s_{l+n}}(\mathbf{z})Y_{-\alpha_{l+n}+s_{l+n}})|0\rangle
\end{align}
Using \eqref{e:sp15} and \eqref{e:sp28}, we have
\begin{align}
\notag&\langle\beta^{sp}|\Gamma_+(\mathbf{x}^{\pm};\mathbf{z})|\alpha^{sp}\rangle\\
\notag&~~~~=\langle\beta^{sp}|\Gamma_+(\mathbf{x}^{\pm})\sum_{\eta=(\eta_1,\dots,\eta_{l+n})}|\eta^{sp}\rangle\langle\eta^{sp}|\Gamma_+(\mathbf{z})|\alpha^{sp}\rangle\\
\notag&~~~~=\sum_{\eta=(\eta_1,\dots,\eta_{l+n})}sp_{\eta/\beta}(\mathbf{x}^{\pm})\langle\eta^{sp}|(\sum^\infty_{s_1=0}h_{s_1}(\mathbf{z})Y_{-\alpha_1+s_1})\cdots (\sum^\infty_{s_{l+n}=0}h_{s_{l+n}}(\mathbf{z})Y_{-\alpha_{l+n}+s_{l+n}})|0\rangle\\
\notag&~~~~=\sum^{\infty}_{s_1,\dots,s_{l+n}=0}h_{s_1}(\mathbf{z})\cdots h_{s_{n+m}}(\mathbf{z})sp_{(\alpha_1-\eta_1,\dots,\alpha_{l+n}-\eta_{l+n})/(\beta_1,\dots,\beta_l)}(\mathbf{x}^{\pm})\\
\notag&~~~~=\det(N_{ij})^{l+n}_{i,j=1},
\end{align}
where \begin{equation}
N_{ij}=\left\{\begin{aligned}
&\sum^\infty_{s_i=0}h_{s_i}(\mathbf{z})h_{\alpha_i-s_i-\beta_j-i+j}(\mathbf{x}^{\pm})~~&~~1\leq j\leq l+1,\\
&\sum^\infty_{s_i=0}h_{s_i}(\mathbf{z})h_{\alpha_i-s_i-\beta_j-i+j}(\mathbf{x}^{\pm})+\sum^\infty_{s_i=0}h_{s_i}(\mathbf{z})h_{\alpha_i-s_i-i-j+2l+2}(\mathbf{x}^{\pm})~&~~l+2\leq j\leq l+n.
\end{aligned}
\right.
\end{equation}
So \eqref{e:sp1} is proved by the generalized Newton identity \eqref{e:sp50}.
\end{proof}
\begin{remark} From the proofs of Theorems \ref{th1} and \ref{th2}, there are two different ways to prove $\langle0|\Gamma_+(\mathbf{x}^{\pm};\mathbf{z})|\lambda^{sp}\rangle=sp_\lambda(\mathbf{x}^{\pm};\mathbf{z})$ for $\lambda=(\lambda_1,\dots,\lambda_n)$.
\end{remark}
\subsection{General branching rule for the universal symplectic functions} Using skew universal symplectic functions, we can obtain a new branching rule.
\begin{theorem}\label{th4}For partition $\lambda=(\lambda_1,\dots,\lambda_n)$, we have
\begin{align}\label{e:sp29}
sp_{\lambda}(\mathbf{x}^{\pm};\mathbf{z})=\sum_{\mu=(\mu_1,\dots,\mu_{n-k})\subset \lambda}sp_{\mu}(\mathbf{\bar{x}}^{\pm})sp_{\lambda/\mu}(\mathbf{\bar{\bar{x}}}^{\pm};\mathbf{z}),
\end{align}
where $\mathbf{\bar{x}}^{\pm}=(x^{\pm}_1,\dots,x_{n-k}^{\pm}),~\mathbf{\bar{\bar{x}}}^{\pm}=(x^{\pm}_{n-k+1},\dots,x_n^{\pm}),~\mathbf{\bar{z}}=(z_1,\dots,z_m)$.
\end{theorem}
\begin{proof} It follows from Theorems \ref{th1} and \ref{th2} that
\begin{align*}
sp_{\lambda}(\mathbf{x}^{\pm};\mathbf{z})=&\langle 0|\Gamma_{+}(\mathbf{x}^{\pm};\mathbf{z})|\lambda^{sp}\rangle\\
=&\langle 0|\Gamma_{+}(\mathbf{\bar{x}}^{\pm})\sum_{\mu=(\mu_1,\dots,\mu_{n-k})}|\mu^{sp}\rangle\langle\mu^{sp}|\Gamma_{+}(\mathbf{\bar{\bar{x}}}^{\pm};\mathbf{z})|\lambda^{sp}\rangle\\
=&\sum_{\mu=(\mu_1,\dots,\mu_{n-k})\subset \lambda}sp_{\mu}(\mathbf{\bar{x}}^{\pm})sp_{\lambda/\mu}(\mathbf{\bar{\bar{x}}}^{\pm};\mathbf{z}).
\end{align*}
Note that the vector $\langle 0|\Gamma_{+}(\bar{\mathbf{x}}^{\pm})$ belongs to
 the subspace spanned by $\langle\alpha^{sp}|$ with $\alpha=(\alpha_1,\dots,\alpha_{n-k})$, on which the operator
 $\sum_{\mu}|\mu^{sp}\rangle\langle\mu^{sp}|$ acts as an identity  (see \eqref{e:sp15}).
\end{proof}
\begin{remark}For $m=0$, \eqref{e:sp29} reduces to the general branching rule for the symplectic Schur functions.
\end{remark}

\subsection{Odd symplectic characters}
Proctor\cite{Pr1988} studied the representation theory of odd symplectic group $Sp_{2n+1}$, and obtained the character $sp_\lambda(\mathbf{x}^{\pm};z)$ of indecomposable $Sp_{2n+1}-$module $V_\lambda$. The characters are called the {\it odd symplectic characters}. Based on the Cauchy-Binet formula and the Cauchy determinant, Okada has given the following bialternant form for $sp_\lambda(\mathbf{x}^{\pm};z)$ \cite[Thm. 1.1]{Ok2020}
\begin{align*}
sp_\lambda(\mathbf{x}^{\pm};z)=\frac{\det(a_{ij})^{n+1}_{i,j=1}}{\det{(b_{ij})^{n+1}_{i,j=1}}}
\end{align*}
with
\begin{equation}
a_{ij}=\left\{\begin{aligned}
&x^{\lambda_j+(n-j+2)}_i-x^{-\lambda_j-(n-j+2)}_i-z^{-1}(x^{\lambda_j+(n-j+1)}_i-x^{-\lambda_j-(n-j+1)}_i)~~&~~1\leq i\leq n,\\
&z^{\lambda_j+n-j+2}-z^{\lambda_j+n-j}~&~~i=n+1,
\end{aligned}
\right.
\end{equation}
and $b_{ij}=x^{n-j+2}_i-x^{-(n-j+2)}_i$ for $1\leq i\leq n$, $b_{n+1,j}=z^{n-j+2}-z^{-(n-j+2)}$.

The following result gives a vertex operator realization of the odd symplectic character $sp_\lambda(\mathbf{x}^{\pm};z)$ and it turns out to be a specialization of the universal symplectic
character $sp_\lambda(\mathbf{x}^{\pm};\mathbf{z})$ when $\mathbf{z}$ is a single parameter $z$ or $m=1$. (This also justifies the notation!)
\begin{theorem}\label{th3}
For general partition $\lambda=(\lambda_1,\dots,\lambda_{n+1})$, we have
\begin{align}\label{e:sp30}
\langle 0|\Gamma_+(\mathbf{x}^{\pm};z)|\lambda^{sp}\rangle~=~sp_\lambda(\mathbf{x}^{\pm};z).
\end{align}
\end{theorem}
\begin{proof} It was known \cite[Prop. 3.2]{JLW2022} that $\langle 0|\Gamma_+(\mathbf{x}^{\pm};x^{\pm}_{n+1})$ is the generating function of the symplectic Schur function $sp_{\nu}(\mathbf{x}^{\pm};x^{\pm}_{n+1})$, i.e.,
\begin{align}
\langle 0|\Gamma_+(\mathbf{x}^{\pm};x^{\pm}_{n+1})~=\sum_{\nu=(\nu_1,\dots,\nu_{n+1})}\langle\nu^{sp}|sp_{\nu}(\mathbf{x}^{\pm};x^{\pm}_{n+1}).
\end{align}

It follows from the Baker-Campbell-Hausdorf formula that
\begin{align}
\exp\left(-\sum^\infty_{n=1}\frac{a_n}{n}z^{-n}\right)Y(w)~=~(1-\frac{w}{z})Y(w)\exp\left(-\sum^\infty_{n=1}\frac{a_n}{n}z^{-n}\right),
\end{align}
i.e. in terms of the components of $Y(w)$
\begin{align}
\exp\left(-\sum^\infty_{n=1}\frac{a_n}{n}z^{-n}\right)Y_i~=~(Y_i-z^{-1}Y_{i+1}) \exp\left(-\sum^\infty_{n=1}\frac{a_n}{n}z^{-n}\right).
\end{align}
Thus
\begin{align}
\exp\left(-\sum^\infty_{n=1}\frac{a_n}{n}z^{-n}\right)|\lambda^{sp}\rangle~=~(Y_{-\lambda_1}-z^{-1}Y_{-\lambda_1+1})\cdots (Y_{-\lambda_{n+1}}-z^{-1}Y_{-\lambda_{n+1}+1})|0\rangle.
\end{align}
Writing $z=x_{n+1}$, we therefore have that
\begin{align}
&\langle 0|\Gamma_+(\mathbf{x}^{\pm};x_{n+1})|\lambda^{sp}\rangle\\
\notag=~&\langle 0|\Gamma_+(\mathbf{x}^{\pm};x^{\pm}_{n+1})\exp\left(-\sum^\infty_{n=1}\frac{a_n}{n}x_{n+1}^{-n}\right)|\lambda^{sp}\rangle\\
\label{e:CB1}=~&\sum_{\nu=(\nu_1,\dots,\nu_{n+1})}\langle\nu^{sp}|sp_\nu(\mathbf{x}^{\pm};x^{\pm}_{n+1})(Y_{-\lambda_1}-x_{n+1}^{-1}Y_{-\lambda_1+1})\cdots (Y_{-\lambda_{n+1}}-x_{n+1}^{-1}Y_{-\lambda_{n+1}+1})|0\rangle\\
\label{e:CB2}=~&\sum_{\substack{\varepsilon_i\in\{0,1\}\\1\leq i\leq n+1}}(-x_{n+1}^{-1})^{|\varepsilon|}sp_{\lambda_1-\varepsilon_1,\dots,\lambda_{n+1}-\varepsilon_{n+1}}(\mathbf{x}^{\pm};x^{\pm}_{n+1})\\
\label{e:CB3}=~&\frac{\det(x^{\lambda_j+(n-j+2)}_i-x^{-\lambda_j-(n-j+2)}_i-x^{-1}_{n+1}(x^{\lambda_j+(n-j+1)}_i-x^{-\lambda_j-(n-j+1)}_i))^{n+1}_{i,j=1}}{\det(x^{n-j+2}_i-x^{-(n-j+2)}_i)^{n+1}_{i,j=1}}\\
\notag=~&sp_\lambda(\mathbf{x}^{\pm};x_{n+1}),
\end{align}
where $|\varepsilon|=\sum_{i}\varepsilon_i$. We remark that if $\lambda_i-\varepsilon_i+1=\lambda_{i+1}-\varepsilon_{i+1}$, then $sp_{\lambda_1-\varepsilon_1,\dots,\lambda_{n+1}-\varepsilon_{n+1}}(\mathbf{x}^{\pm};x^{\pm}_{n+1})=0$. In other words, $sp_{\lambda_1-\varepsilon_1,\dots,\lambda_{n+1}-\varepsilon_{n+1}}(\mathbf{x}^{\pm};x^{\pm}_{n+1})$ makes sense only when $(\lambda_1-\varepsilon_1,\dots,\lambda_{n+1}-\varepsilon_{n+1})$ is also a partition.
\end{proof}
\begin{remark} 
We have mentioned that the odd symplectic character $sp_\lambda(\mathbf{x}^{\pm};z)$ is a special case $(m=1)$ of the universal symplectic functions
$sp_\lambda(\mathbf{x}^{\pm};\mathbf{z})$ ($\mathbf{z}$ is an $m$-tuple parameter).
Taking $z=1$ in Theorem \ref{th3}, one recovers Proctor's bialternant formula for $sp_\lambda(\mathbf{x}^{\pm};1)$ \cite{Ok2020,Pr1988}. When $z=-1$ in Theorem \ref{th3}, one also obtains Krattenthaler's bideterminantal formula for $sp_\lambda(\mathbf{x}^{\pm};-1)$\cite{Ok2020,Kr1995}.
\end{remark}

As a consequence of Theorem \ref{th1}, we have the {\it Jacobi-Trudi identity} for $sp_\lambda(\mathbf{x}^{\pm};z)$:
\begin{corollary} For general partition $\lambda=(\lambda_1, \ldots, \lambda_{n+1})$,
\begin{align*}
sp_\lambda(\mathbf{x}^{\pm};z)=\det\big{(}h_{\lambda_i-i+j}(\mathbf{x}^{\pm};z)+\delta_{j>1}h_{\lambda_i-i-j+2}(\mathbf{x}^{\pm};z)\big{)}^{n+1}_{ i,j=1}.
\end{align*}
\end{corollary}
The proof of  Theorem \ref{th3} also provides a {\it transition formula} from the symplectic Schur functions to
the odd symplectic characters.
\begin{theorem}\label{th10}Let $\mathbf{x}^{\pm}=(x^\pm_1,\dots,x^\pm_n)$ and $\mathbf{\underline{x}}^{\pm}=(x^\pm_1,\dots,x^\pm_n,z^{\pm})$. Then
\begin{align}
sp_\lambda(\mathbf{x}^{\pm};z)=\sum_{\substack{\varepsilon_i\in\{0,1\}\\1\leq i\leq n+1}}(-z^{-1})^{|\varepsilon|}sp_{\lambda_1-\varepsilon_1,\dots,\lambda_{n+1}-\varepsilon_{n+1}}(\mathbf{\underline{x}}^{\pm}).
\end{align}
\end{theorem}

 According to the Gelfand-Tsetlin representation of $sp_{\mu}(\mathbf{x}^{\pm})$ (cf. \cite[Thm. 4.1]{JLW2022}, \cite[Thm. 2.6]{AF2020}), we
 also obtain a Gelfand-Tsetlin representation for the odd symplectic characters (see \cite[(2.6)]{WW2009}).
\begin{theorem}For partition $\lambda=(\lambda_1,\dots,\lambda_n)$,
\begin{align}\label{e:sp11}
sp_{\lambda}(\mathbf{x}^{\pm};x_{n+1})=\sum_{\emptyset=z_0\prec z_1\prec\dots\prec z_{2n}\prec z_{2n+1}=\lambda}x^{|z_{2n+1}|-|z_{2n}|}_{n+1}\prod^n_{i=1}x^{2|z_{2i-1}|-|z_{2i}|-|z_{2i-2}|}_i,
\end{align}
where partitions $z_k=(z_{k,1},\dots,z_{k,\lceil\frac{ k}{2}\rceil})$ satisfy the symplectic Gelfand-Tsetlin pattern
\begin{align*}
z_{k+1,j}\leq z_{k,j-1}\leq z_{k+1,j-1}~~~~\text{for}~~~~ 2\leq j\leq \lceil\frac{k+1}{2}\rceil.
\end{align*}
\end{theorem}
From Theorem \ref{th2}, we also have a Gelfand-Tsetlin representation for skew odd symplectic character.
\begin{corollary}For partitions $\mu=(\mu_1,\dots,\mu_l)\subseteq\lambda=(\lambda_1,\dots,\lambda_{l+n})$ and $\mathbf{x}^{\pm}=(x^{\pm}_1,\dots,x_n^{\pm})$, we have
\begin{align*}
sp_{\lambda/\mu}(\mathbf{x}^{\pm};x_{n+1})=\sum_{\mu=z_0\prec z_1\prec\dots\prec z_{2n}\prec z_{2n+1}=\lambda}x^{|z_{2n+1}|-|z_{2n}|}_{n+1}\prod^n_{i=1}x^{2|z_{2i-1}|-|z_{2i}|-|z_{2i-2}|}_i.
\end{align*}
where partitions $z_k=(z_{k,1},\dots,z_{k,l+\frac{\lceil k\rceil}{2}})$ satisfy the symplectic Gelfand-Tsetlin pattern
\begin{align*}
z_{k+1,j}\leq z_{k,j-1}\leq z_{k+1,j-1}~~~~\text{for}~~~~ 2\leq j\leq l+\frac{\lceil k+1\rceil}{2}.
\end{align*}
\end{corollary}


\subsection{Universal symplectic characters}
Koike and Terada have defined the universal symplectic character for the symplectic groups using Weyl's Jacobi-Trudi identity \cite[Thm. 1.3.3]{KT1987}
\begin{align}
sp_\lambda(\mathbf{z})=\det\big{(}h_{\lambda_i-i+j}(\mathbf{z})+\delta_{j>1}h_{\lambda_i-i-j+2}(\mathbf{z})\big{)}^{m}_{ i,j=1}.
\end{align}
Therefore the universal symplectic character is a special case ($n=0$) of our universal symplectic function.

By Theorem \ref{th1} we have the following vertex operator realizations
\begin{align}
&\langle0|\Gamma_+(\mathbf{z})|\lambda^{sp}\rangle=sp_\lambda(\mathbf{z})
\end{align}
for the general partition $\lambda=(\lambda_1,\dots,\lambda_m)$.

%
\subsection{Intermediate symplectic characters}For $\lambda=(\lambda_1,\dots,\lambda_{n+1},\underbrace{0,\dots,0}_{m-1})$, the universal symplectic functions \eqref{e:uni} reduce to intermediate symplectic characters (see \cite[Prop. 2.5]{Ok2021}, \cite[(3.1)]{Kr1995}). The latter were used by Proctor \cite{Pr1991} to define the so-called ``trace-free'' characters of the intermediate symplectic groups.

\section{Universal orthogonal functions}
\subsection{Orthogonal vertex operators}
 We can similarly develop the vertex algebraic approach to the orthogonal Schur functions and their universal generalizations. Define the vertex operators\cite{JN2015,Ba1996, SZ2006}
\begin{align*}
&W(z)=(1-z^2)\exp\left(\sum^\infty_{n=1}\frac{a_{-n}}{n}z^n\right)\exp\left(-\sum^\infty_{n=1}\frac{a_n}{n}(z^{-n}+z^n)\right)=\sum_{n\in \mathbb{Z}}W_nz^{-n},\\
&W^*(z)=\exp\left(-\sum^\infty_{n=1}\frac{a_{-n}}{n}z^n\right)\exp\left(\sum^\infty_{n=1}\frac{a_n}{n}(z^{-n}+z^n)\right)=\sum_{n\in \mathbb{Z}}W^*_nz^n.
\end{align*}
The following commutation relations are easy consequences of the vertex operator calculus \cite{Ba1996,JN2015}:
\begin{equation}
\begin{aligned}\label{e:o15}
&W_iW_j+W_{j+1}W_{i-1}=0,\qquad W^*_iW^*_j+W^*_{j-1}W^*_{i+1}=0,\qquad W_iW^*_j+W^*_{j+1}W_{i+1}=\delta_{i,j},\\
&W_n|0\rangle=W^*_{-n}|0\rangle=0,\quad \text{for}~ n>0,~~~~~W_0|0\rangle=W^*_0|0\rangle=|0\rangle\\
&\langle 0|W_n=-\langle 0|W_{-n-2},\qquad \langle 0|W^*_n=\langle 0|W^*_{-n}.
\end{aligned}
\end{equation}

For general partition $\lambda=(\lambda_1,\dots,\lambda_l)$, introduce
\begin{align*}
|\lambda^{o}\rangle=W_{-\lambda_1}W_{-\lambda_2}\cdots W_{-\lambda_l}|0\rangle, \qquad\langle \lambda^{o}|=\langle 0|W^*_{-\lambda_l}\cdots W^*_{-\lambda_1}.
\end{align*}
From \eqref{e:o15}, one has the following equations for $\mu=(\mu_1,\dots,\mu_l)$ and permutation $\sigma\in S_l$
\begin{align}
&\varepsilon(\sigma)\langle 0|\left(^{W^*_{-\mu_{\sigma(l)}+\sigma(l)-l}}_{\delta_lW^*_{\mu_{\sigma(l)}-\sigma(l)+l}}\right)\cdots \left(^{W^*_{-\mu_{\sigma(i)}+\sigma(i)-i}}_{\delta_iW^*_{\mu_{\sigma(i)}-\sigma(i)+2l-i}}\right)\cdots \left(^{W^*_{-\mu_{\sigma(1)}+\sigma(1)-1}}_{\delta_1W^*_{\mu_{\sigma(1)}-\sigma(1)+2l-1}}\right)=\langle\mu^{o}|,\\
&\varepsilon(\sigma)W_{-\mu_{\sigma(l)}+\sigma(l)-l}\dots W_{\mu_{\sigma(l)}-\sigma(l)+l}|0\rangle=|\mu^{o}\rangle,
\end{align}
where $\delta_i$ denotes $\delta_{\mu_l\neq 0}\delta_{\sigma(i)\neq l}$.

It was shown \cite[Theorem 2.1]{JLW2022} that
\begin{align}\label{e:o12}
\langle \mu^{o}|\lambda^{o}\rangle=\delta_{\lambda\mu}
\end{align}for general partitions $\mu=(\mu_1,\mu_2,\dots,\mu_l)$ and $\lambda=(\lambda_1,\lambda_2,\dots,\lambda_l)$ with zero parts in the tails allowed. Therefore, for general partition $\alpha=(\alpha_1,\dots,\alpha_n)$, we have
\begin{align}\label{e:sp15b}
\langle \alpha^{o}|\sum_{\eta=(\eta_1,\dots,\eta_n)}|\eta^{o}\rangle\langle\eta^{o}|=\langle \alpha^{o}|,~~\sum_{\eta=(\eta_1,\dots,\eta_n)}|\eta^{o}\rangle\langle\eta^{o}||\alpha^{o}\rangle=|\alpha^{o}\rangle.
\end{align}
Here again the operator $\sum_{\eta}|\eta^{o}\rangle\langle\eta^{o}|$ can be viewed as an identity operator on the subspace spanned by $|\alpha^o\rangle$'s
(or $\langle\alpha^{o}|$'s).


\subsection{(Skew) orthogonal Schur functions}For each partition $\lambda=(\lambda_1,\lambda_2,\ldots,\lambda_l)$, $\lambda_l\geq 0$, the (even) orthogonal Schur functions\cite{Kr1995,Mac1995,AK2022,Me2019} is defined by
\begin{align*}
o_\lambda(\mathbf{x}^{\pm})=2\frac{\det(x^{\lambda_j+(l-j)}_i+\delta_{j\neq l}\delta_{\lambda_l\neq 0}x^{-\lambda_j-(l-j)}_i)^l_{i,j=1}}{\det(x^{l-j}_i+x^{-(l-j)}_i)^l_{i,j=1}}.
\end{align*}
In particular, for partition $\lambda=(\lambda_1,\lambda_2,\ldots,\lambda_l)$ with $\lambda_l=0$,
\begin{align}\label{e:o9}
o_\lambda(\mathbf{x}^{\pm})=\frac{\det(x^{\lambda_j+(l-j)}_i+x^{-\lambda_j-(l-j)}_i)^l_{i,j=1}}{\det(x^{l-j}_i+x^{-(l-j)}_i)^l_{i,j=1}}.
\end{align}
The following vertex operator realizations for (skew) orthogonal Schur functions can be found in \cite[Proposition 3.3]{JLW2022}.
\begin{theorem}\label{th8}For general partitions $\nu=(\nu_1,\dots,\nu_{n})$, $\beta=(\beta_1,\dots,\beta_l)$ and
$\alpha=(\alpha_1,\dots,\alpha_{l+n})$,
one has
\begin{align}
&\notag\langle0|\Gamma_+(\mathbf{x}^{\pm})|\nu^{o}\rangle=o_\nu(\mathbf{x}^{\pm})=\det\big{(}h_{\lambda_i-i+j}(\mathbf{x}^{\pm})-h_{\lambda_i-i-j}(\mathbf{x}^{\pm})\big{)}^{n}_{ i,j=1},\\
&\langle\beta^{o}|\Gamma_+(\mathbf{x}^{\pm})|\alpha^{o}\rangle=\det\big{(}h_{\alpha_i-\beta_j-i+j}(\mathbf{x}^{\pm})-\delta_{j>l}h_{\alpha_i-i-j+2l}(\mathbf{x}^{\pm})\big{)}^{l+n}_{ i,j=1}=o_{\alpha/\beta}(\mathbf{x}^{\pm}),
\end{align}
which is zero unless $\beta\subset\alpha$ and
$o_{\alpha/\beta}(\mathbf{x}^{\pm})$ is the skew orthogonal Schur function.
\end{theorem}
In the same work \cite{JLW2022}, the general branching rule for orthogonal Schur functions was shown:
\begin{align}\label{e:br2}
o_\nu(x^\pm_1,\dots,x^\pm_n)=\sum_{\gamma=(\gamma_1,\dots,\gamma_{n-k})\subset\nu}o_\gamma(x^\pm_1,\dots,x^\pm_{n-k})o_{\nu/\gamma}(x^\pm_{n-k+1},\dots,x^\pm_n).
\end{align}
\subsection{(Skew) universal orthogonal functions} Similar to the universal symplectic functions, we introduce the {\it universal orthogonal functions} (invariants of $(S_n\ltimes (\mathbb{Z}_2)^n)\times S_m$)
\begin{align}\label{e:uni1}
o_\lambda(\mathbf{x}^{\pm};\mathbf{z})=\det\big{(}h_{\lambda_i-i+j}(\mathbf{x}^{\pm};\mathbf{z})-h_{\lambda_i-i-j}(\mathbf{x}^{\pm};\mathbf{z})\big{)}^{n+m}_{ i,j=1}.
 \end{align}
 Using the same method in proving Theorem \ref{th1}, we obtain the following vertex operator realization.
 \begin{theorem}\label{th9}
For general partition $\lambda=(\lambda_1,\dots,\lambda_{n+m})$.
one has
\begin{align}\label{e:o14}
\langle0|\Gamma_+(\mathbf{x}^{\pm};\mathbf{z})|\lambda^{o}\rangle=o_{\lambda}(\mathbf{x}^{\pm};\mathbf{z}).
\end{align}
\end{theorem}

\begin{theorem}
For general partitions $\beta=(\beta_1,\dots,\beta_l)$ and $\alpha=(\alpha_1,\dots,\alpha_{l+n})$, one has
\begin{align}
\langle\beta^{o}|\Gamma_+(\mathbf{x}^{\pm};\mathbf{z})|\alpha^{o}\rangle=&\det\big{(}h_{\alpha_i-\beta_j-i+j}(\mathbf{x}^{\pm};\mathbf{z})--\delta_{j>l}h_{\alpha_i-i-j+2l}(\mathbf{x}^{\pm};\mathbf{z})\big{)}^{l+n}_{ i,j=1}\\
=&o_{\alpha/\beta}(\mathbf{x}^{\pm};\mathbf{z}),
\end{align}
which is zero unless $\beta\subset\alpha$ and will be called 
the skew universal orthogonal function associated to $\alpha/\beta$.
\end{theorem}
We thus have the following {\it general branching rule} for universal orthogonal function
\begin{align}
o_\nu(x^\pm_1,\dots,x^\pm_n;\mathbf{z})=\sum_{\gamma=(\gamma_1,\dots,\gamma_{n-k})\subset\nu}o_\gamma(x^\pm_1,\dots,x^\pm_{n-k})o_{\nu/\gamma}(x^\pm_{n-k+1},\dots,x^\pm_n;\mathbf{z}).
\end{align}

\subsection{Odd orthogonal Schur functions}Odd orthogonal Schur functions $so_\lambda(\mathbf{x}^{\pm})$ can be used to define irreducible
characters of $SO_{2n+1}$, and they have the bialternant form (cf. \cite{Wey1946})
\begin{align}
so_\lambda(\mathbf{x}^{\pm})=\frac{\det(x^{\lambda_j+n-j+\frac{1}{2}}_i-x^{-(\lambda_j+n-j+\frac{1}{2})}_i)^{n}_{i,j=1}}{\det(x^{n-j+\frac{1}{2}}_i-x^{-(n-j+\frac{1}{2})}_i)^{n}_{i,j=1}}.
\end{align}

We now consider the vertex operator realization for $so_\lambda(\mathbf{x}^{\pm})$, which turns out to be a special universal orthogonal function.
The following is proved by a similar method as in Theorem \ref{e:sp30}, nevertheless we will show its
special case below (for $z=\pm 1$) in Corollary \ref{c:orth} and Corollary \ref{c:oddorth}.
\begin{theorem}\label{th30} For any partition $\lambda=(\lambda_1,\dots,\lambda_n,0)$, one has that
\begin{align}\label{e:oo4}
\langle0|\Gamma_+(\mathbf{x}^{\pm};z)|\lambda^{o}\rangle=o_\lambda(\mathbf{x}^{\pm};z)=\frac{\det(A_{ij})^{n+1}_{i,j=1}}{\det(B_{ij})^{n+1}_{i,j=1}},
\end{align}
where
\begin{equation}
A_{ij}=\left\{\begin{aligned}
&x^{\lambda_j+(n-j+1)}_i+x^{-\lambda_j-(n-j+1)}_i-\delta_{\lambda_j>0}z^{-1}(x^{\lambda_j+(n-j)}_i+x^{-\lambda_j-(n-j)}_i)~~&~~1\leq i\leq n,\\
&\delta_{\lambda_j>0}(z^{\lambda_j+n-j+1}-z^{\lambda_j+n-j-1})+\delta_{\lambda_j,0}(z^{n+1-j}+z^{-n-i+j})~&~~i=n+1,
\end{aligned}
\right.
\end{equation}
and $B_{ij}=x^{n-j+1}_i+x^{-(n-j+1)}_i$ for $1\leq i\leq n$, $B_{n+1,j}=z^{n-j+1}+z^{-(n-j+1)}$.
\end{theorem}

Similar to Theorem \ref{th10}, one can also express this special orthogonal Schur functions in terms of (even) orthogonal Schur functions as follows.
\begin{align}
o_\lambda(\mathbf{x}^{\pm};z)=\sum_{\substack{\varepsilon_i\in\{0,1\}\\1\leq i\leq n}}(-z^{-1})^{|\varepsilon|}o_{\lambda_1-\varepsilon_1,\dots,\lambda_{n}-\varepsilon_{n}}(\mathbf{\underline{x}}^{\pm}),
\end{align}
where $\mathbf{\underline{x}}^{\pm}=(x^\pm_1,\dots,x^\pm_n,z^{\pm})$ and $\lambda=(\lambda_1,\dots,\lambda_n,0)$.
When $z$ takes some special values, $o_\lambda(\mathbf{x}^{\pm};z)$ have representation-theoretic interpretations.
\begin{corollary}\label{c:orth} Let $\lambda$ be a partition with $n$ parts. $o_\lambda(\mathbf{x}^{\pm};1)$ is given by
\begin{align}\label{e:oo2}
o_\lambda(\mathbf{x}^{\pm};1)=\frac{\det(x^{\lambda_j+n-j+\frac{1}{2}}_i-x^{-(\lambda_j+n-j+\frac{1}{2})}_i)^{n}_{i,j=1}}{\det(x^{n-j+\frac{1}{2}}_i-x^{-(n-j+\frac{1}{2})}_i)^{n}_{i,j=1}}=so_\lambda(\mathbf{x}^{\pm}).
\end{align}
\end{corollary}
\begin{proof}
When $z=1$, \eqref{e:oo4} becomes
\begin{align}
o_\lambda(\mathbf{x}^{\pm};1)=\frac{\det(A_{ij})^{n+1}_{i,j=1}}{\det(B_{ij})^{n+1}_{i,j=1}}
\end{align}
with
\begin{equation}
A_{ij}=\left\{\begin{aligned}
&(x^{\lambda_j+(n-j+\frac{1}{2})}_i-x^{-\lambda_j-(n-j+\frac{1}{2})}_i)(x^{\frac{1}{2}}_i-x^{-\frac{1}{2}}_i)~~&~~1\leq i,j\leq n,\\
&0~~&~~i=n+1\& 1\leq j\leq n,\\
&2~~&~~i=n+1\& j=n+1
\end{aligned}
\right.
\end{equation}
and $B_{ij}=x^{n-j+1}_i+x^{-(n-j+1)}_i$ for $1\leq i\leq n$, $B_{n+1,j}=2$.
Using general properties of determinants, we have
\begin{align*}
\det(A_{ij})^{n+1}_{i,j=1}=&2\prod^n_{i=1}(x^{\frac{1}{2}}_i-x^{-\frac{1}{2}}_i)\det(x^{\lambda_j+(n-j+\frac{1}{2})}_i-x^{-\lambda_j-(n-j+\frac{1}{2})}_i)^{n}_{i,j=1}.
\end{align*}

Subtract the $i$th column from the ($i-1$)st column for $i = n+1, n, . . . , 2$, we see that
\begin{align*}
\det(B_{ij})^{n+1}_{i,j=1}=&2\det(x^{n-j+1}_i+x^{-(n-j+1)}_i-x^{n-j}_i-x^{-(n-j)}_i)^{n}_{i,j=1}\\
=&2(x^{\frac{1}{2}}_i-x^{-\frac{1}{2}}_i)\det(x^{n-j+\frac{1}{2}}_i-x^{-(n-j+\frac{1}{2})}_i)^{n}_{i,j=1}.
\end{align*}
Then the ratio of $\det(A_{ij})^{n+1}_{i,j=1}$ by $\det(B_{ij})^{n+1}_{i,j=1}$ gives \eqref{e:oo2}.
\end{proof}
\begin{remark} As we have seen that
the odd orthogonal Schur function is a special case of universal orthogonal functions. Using \eqref{e:oo4}, we have the  vertex operator realization for odd orthogonal Schur function by $\langle0|\Gamma_+(\mathbf{x}^{\pm};1)|\lambda^{o}\rangle=so_\lambda(\mathbf{x}^{\pm})$ with $l(\lambda)=n$. In the forthcoming paper, we will offer another vertex operator representation of odd orthogonal Schur functions.
\end{remark}
\begin{corollary}\label{c:oddorth} Let $\lambda=(\lambda_1,\dots,\lambda_n,0)$ be a partition with $n$ parts. $o_\lambda(\mathbf{x}^{\pm};-1)$ is given by
\begin{align}\label{e:oo3}
o_\lambda(\mathbf{x}^{\pm};-1)=\frac{\det(x^{\lambda_j+n-j+\frac{1}{2}}_i+x^{-(\lambda_j+n-j+\frac{1}{2})}_i)^{n}_{i,j=1}}{\det(x^{n-j+\frac{1}{2}}_i+x^{-(n-j+\frac{1}{2})}_i)^{n}_{i,j=1}}
\end{align}
\end{corollary}
\begin{proof}
When $z=-1$, \eqref{e:oo4} becomes
\begin{align}
o_\lambda(\mathbf{x}^{\pm};-1)=\frac{\det(A_{ij})^{n+1}_{i,j=1}}{\det(B_{ij})^{n+1}_{i,j=1}}
\end{align}
with
\begin{equation}
A_{ij}=\left\{\begin{aligned}
&(x^{\lambda_j+(n-j+\frac{1}{2})}_i+x^{-\lambda_j-(n-j+\frac{1}{2})}_i)(x^{\frac{1}{2}}_i+x^{-\frac{1}{2}}_i)~~&~~1\leq i,j\leq n,\\
&0~~&~~i=n+1\& 1\leq j\leq n,\\
&2~~&~~i=n+1\& j=n+1
\end{aligned}
\right.
\end{equation}
and $B_{ij}=x^{n-j+1}_i+x^{-(n-j+1)}_i$ for $1\leq i\leq n$, $B_{n+1,j}=2(-1)^{n+1-j}$.
Thus
\begin{align*}
\det(A_{ij})^{n+1}_{i,j=1}=&2\prod^n_{i=1}(x^{\frac{1}{2}}_i+x^{-\frac{1}{2}}_i)\det(x^{\lambda_j+(n-j+\frac{1}{2})}_i+x^{-\lambda_j-(n-j+\frac{1}{2})}_i)^{n}_{i,j=1}.
\end{align*}
By adding the $i$th column to the ($i-1$)st column for $i = n+1, n, . . . , 2$,
\begin{align*}
\det(B_{ij})^{n+1}_{i,j=1}=&2\det(x^{n-j+1}_i+x^{-(n-j+1)}_i+x^{n-j}_i+x^{-(n-j)}_i)^{n}_{i,j=1}\\
=&2(x^{\frac{1}{2}}_i+x^{-\frac{1}{2}}_i)\det(x^{n-j+\frac{1}{2}}_i+x^{-(n-j+\frac{1}{2})}_i)^{n}_{i,j=1}.
\end{align*}
Then the quotient $\frac{\det(A_{ij})^{n+1}_{i,j=1}}{\det(B_{ij})^{n+1}_{i,j=1}}$ equals to \eqref{e:oo3}.
\end{proof}
\begin{remark}
$o_\lambda(\mathbf{x}^{\pm};-1)$ is the usual odd orthogonal character evaluated as the ``negative'' part of the orthogonal group\cite[(3.6)]{Kr1995}.
\end{remark}
\subsection{Universal orthogonal characters}In \cite[Theorem 1.3.2]{KT1987}, Koike and Terada introduced the universal orthogonal character $o_\lambda(\mathbf{z})$ of $SO_{2m}$ in the ring of symmetric functions, and they gave the determinantal expression
\begin{align}
o_\lambda(\mathbf{z})=\det\big{(}h_{\lambda_i-i+j}(\mathbf{z})-h_{\lambda_i-i-j}(\mathbf{z})\big{)}^{m}_{ i,j=1}.
\end{align}

 Clearly this is a special case of the universal orthogonal function \eqref{e:uni1} with $n=0$. Consequently we get vertex operator realizations for universal orthogonal characters and skew universal orthogonal characters by using Theorem \ref{th9}. We record this special case as follows.
\begin{theorem}
For general partition $\nu=(\nu_1,\dots,\nu_{m})$,
one has
\begin{align}
\langle0|\Gamma_+(\mathbf{z})|\nu^{o}\rangle=o_{\nu}(\mathbf{z}).
\end{align}
\end{theorem}

\subsection{Intermediate orthogonal characters}
In \cite[(3.2)]{Kr1995}, Krattenthaler introduced  the intermediate orthogonal character
\begin{align}
\det\big{(}h^{\prime}_{\lambda_i-i+j}(\mathbf{x}^{\pm};\mathbf{z})+\delta_{j>1}h^{\prime}_{\lambda_i-i-j+2}(\mathbf{x}^{\pm};\mathbf{z})\big{)}^n_{i,j=1}
\end{align}
for $\lambda=(\lambda_1,\dots,\lambda_n)$ and $h^{\prime}_k(\mathbf{x}^{\pm};\mathbf{z})=h_k(\mathbf{x}^{\pm};\mathbf{z})-h_{k-2}(\mathbf{x}^{\pm};\mathbf{z})$. In fact, for $\lambda=(\lambda_1,\dots,\lambda_n,\underbrace{0,\dots,0}_m)$, the universal orthogonal function \eqref{e:uni1} reduces to
\begin{align}\label{e:o13}
o_\lambda(\mathbf{x}^{\pm};\mathbf{z})=\det\big{(}h_{\lambda_i-i+j}(\mathbf{x}^{\pm};\mathbf{z})-h_{\lambda_i-i-j}(\mathbf{x}^{\pm};\mathbf{z})\big{)}^{n}_{ i,j=1}.
\end{align}
Subtracting the $i$th column from the ($i+2$)st column for $i = n-2, n-3, . . . , 1$, \eqref{e:o13} becomes
\begin{align}
o_\lambda(\mathbf{x}^{\pm};\mathbf{z})=\det\big{(}h^{\prime}_{\lambda_i-i+j}(\mathbf{x}^{\pm};\mathbf{z})+\delta_{j>1}h^{\prime}_{\lambda_i-i-j+2}(\mathbf{x}^{\pm};\mathbf{z})\big{)}^n_{i,j=1},
\end{align}
therefore the intermediate orthogonal character is also a special case of the universal orthogonal function \eqref{e:uni1}.

\section{Interpolating Schur functions}
Bisi and Zygouras \cite{BZ2022} have introduced $CB$-interpolating Schur functions $s^{CB}_\lambda(x;\beta)$ that interpolate between characters of
types $C$ and $B$, and $DB$-interpolating Schur functions $s^{DB}_\lambda(x;\alpha)$ that interpolate characters of
types $D$ and $B$. We now discuss their vertex operator realizations and applications\footnote{Our version of $DB$-interpolating Schur functions $s^{DB}_\lambda(x;\beta)$ are in fact different from those of \cite{BZ2022} but with favorite properties in
view of vertex realizations (see Remark \ref{DBr1}).}.

\subsection{$CB$-interpolating Schur functions $s^{CB}_\lambda(x;\beta)$} The specializations $\beta=0, 1$ of $s^{CB}_\lambda(x;\beta)$ are respectively the symplectic Schur functions $sp_\lambda(\mathbf{x}^{\pm})$ and orthogonal Schur functions $so_\lambda(\mathbf{x}^{\pm})$. Let
\begin{align*}
\overline{\Gamma}_+(\mathbf{x}^{\pm};\beta)=\Gamma_+(\mathbf{x}^{\pm})\exp\left(-\sum^\infty_{n=1}\frac{a_n}{n}(-\beta)^n\right),
\end{align*}
where $\mathbf{x}^{\pm}=(x^{\pm}_1,\dots,x_n^{\pm})$.
\begin{theorem}\label{th20}For a partition $\lambda=(\lambda_1,\dots,\lambda_n)$,
\begin{eqnarray}
\label{e:CB4}\langle0|\overline{\Gamma}_+(\mathbf{x}^{\pm};\beta)|\lambda^{sp}\rangle&=&\sum_{\substack{\varepsilon_i\in\{0,1\}\\1\leq i\leq n}}\beta^{|\varepsilon|}sp_{\lambda_1-\varepsilon_1,\dots,\lambda_{n}-\varepsilon_{n}}(\mathbf{x}^{\pm})\\
\label{e:CB5}&=&\frac{\det(x^{\lambda_j+(n-j+1)}_i-x^{-\lambda_j-(n-j+1)}_i+\beta(x^{\lambda_j+(n-j)}_i-x^{-\lambda_j-(n-j)}_i))^{n}_{i,j=1}}{\det(x^{n-j+1}_i-x^{-(n-j+1)}_i)^{n}_{i,j=1}}
\end{eqnarray}
\end{theorem}
\begin{proof} Using the equations \eqref{e:CB1}--\eqref{e:CB3} in the proof of Theorem \ref{th3}, we have that
\begin{align}
\notag&\langle0|\overline{\Gamma}_+(\mathbf{x}^{\pm};\beta)|\lambda^{sp}\rangle\\
\notag=~&\sum_{\nu=(\nu_1,\dots,\nu_{n})}\langle\nu^{sp}|sp_\nu(\mathbf{x}^{\pm})(Y_{-\lambda_1}+\beta Y_{-\lambda_1+1})\cdots (Y_{-\lambda_{n}}+\beta Y_{-\lambda_{n}+1})|0\rangle\\
\label{e:CB8}=~&\sum_{\substack{\varepsilon_i\in\{0,1\}\\1\leq i\leq n}}\beta^{|\varepsilon|}sp_{\lambda_1-\varepsilon_1,\dots,\lambda_{n}-\varepsilon_{n}}(\mathbf{x}^{\pm})\\
\notag=~&\frac{\det(x^{\lambda_j+(n-j+1)}_i-x^{-\lambda_j-(n-j+1)}_i+\beta(x^{\lambda_j+(n-j)}_i-x^{-\lambda_j-(n-j)}_i))^{n}_{i,j=1}}{\det(x^{n-j+1}_i-x^{-(n-j+1)}_i)^{n}_{i,j=1}}.
\end{align}
\end{proof}
\begin{remark}\label{CBre1}
It is easy to see that $\langle0|\overline{\Gamma}_+(\mathbf{x}^{\pm};0)|\lambda^{sp}\rangle=sp_\lambda(\mathbf{x}^{\pm})$ and $\langle0|\overline{\Gamma}_+(\mathbf{x}^{\pm};1)|\lambda^{sp}\rangle=so_\lambda(\mathbf{x}^{\pm})$ by the properties of determinant. Theorem \ref{th20}
 gives a vertex operator realization for the $CB$-interpolating Schur functions $s^{CB}_\lambda(x;\beta)$ defined in \cite[(1.1),(3.8)]{BZ2022}.
\end{remark}
Using the method of proving Theorem \ref{th1}, we also obtain the Jacobi--Trudi formula for $CB$-interpolating Schur functions.
\begin{corollary} Let $h_n(x;\beta)$ be the generalized homogeneous symmetric functions defined by $(1+\beta w)\prod^n_{i=1}\frac{1}{(1-x_iw)(1-x^{-1}_iw)}=\sum_{i\in \mathbb{Z}}h_i(x;\beta)w^i$, then
\begin{align}\label{e:CB7}
s^{CB}_\lambda(x;\beta)=\det\left(h_{\lambda_i-i+j}(x;\beta)+\delta_{j>1}h_{\lambda_i-i-j+2}(x;\beta)\right)^n_{ i,j=1}.
\end{align}
\end{corollary}

\begin{remark}\label{CBre2}
In fact, the determinantal identity \eqref{e:CB7} of $s^{CB}_\lambda(x;\beta)$ coincides with the Jacobi--Trudi identity for the orthosymplectic Schur function $spo_\lambda(x/\beta)$\cite[Cor 4.3]{SV2016}.
\end{remark}
We now offer a short proof for the Brent-Krattenthaler-Warnaar-type identity \eqref{BKW}.
\begin{corollary}(\cite[Thm 4.7]{Kum2024})
Let $m$ and $n$ be positive integers with $m\leq n$, $u$ a nonnegative integer, and $x=(x_1,\dots,x_n)$ and $y=(y_1,\dots,y_m)$. Then
\begin{equation}\begin{aligned}
     &\sum_{\lambda} \beta^u spo_{(u^{n-m},\lambda)}(x/\beta) spo_{\lambda}(y/\beta) \\
     &=\sum_{j=1}^{m+n+1} \beta^{u+m+n+1-j} sp_{\big(u^{j-1},(u-1)^{m+n+1-j}\big)}(\mathbf{x}^{\pm},\mathbf{y}^{\pm}),
     \end{aligned}
\end{equation}
where $\lambda$ runs over partitions of length at most $m$ such that $\lambda_1\leq u$.
\end{corollary}
\begin{proof} It follows from Theorem \ref{th20} and Remark \ref{CBre1} that
\begin{align}
\notag s^{CB}_{u^{m+n}}(x,y;\beta)=~&\langle0|\overline{\Gamma}_+(\mathbf{x}^{\pm},\mathbf{y}^{\pm};\beta)|(u^{m+n})^{sp}\rangle\\
\notag=~&\sum_{\substack{\varepsilon_i\in\{0,1\}\\1\leq i\leq m+n}}\beta^{|\varepsilon|}sp_{u-\varepsilon_1,\dots,u-\varepsilon_{m+n}}(\mathbf{x}^{\pm},\mathbf{y}^{\pm})\\
\label{e:CB9}=~&\sum_{j=1}^{m+n+1} \beta^{m+n+1-j} sp_{\big(u^{j-1},(u-1)^{m+n+1-j}\big)}(\mathbf{x}^{\pm},\mathbf{y}^{\pm}),
\end{align}
where the third equations is due to the fact $sp_{u-\varepsilon_1,\dots,u-\varepsilon_{m+n}}(\mathbf{x}^{\pm},\mathbf{y}^{\pm})=0$ for $\varepsilon_i=1$ and $\varepsilon_{i+1}=0$. Combining with identity \eqref{e:CB6}
\begin{align}
s^{CB}_{u^{m+n}}(x,y;\beta)=\sum_{\lambda_1\leq u}s^{CB}_{(u^{n-m},\lambda)}(x;\beta)s^{CB}_{\lambda}(y;\beta)
\end{align}
and Remark \ref{CBre2}, we have
\begin{align}\label{e:CB10}
s^{CB}_{u^{m+n}}(x,y;\beta)=\sum_{\lambda_1\leq u}spo_{(u^{n-m},\lambda)}(x/\beta) spo_{\lambda}(y/\beta).
\end{align}
In view of \eqref{e:CB9} and \eqref{e:CB10}, we can finish the proof.
\end{proof}

\subsection{$DB$-interpolating Schur functions $s^{DB}_\lambda(x;\beta)$} We now introduce a family of symmetric functions $s^{DB}_\lambda(x;\beta)$
that interpolate between the even orthogonal Schur functions and odd orthogonal Schur functions:
Namely, when $\beta=0$ they are the orthogonal Schur functions, while
$s^{DB}_\lambda(x;\beta)|_{\beta=1}$ are the odd orthogonal Schur functions.

We proceed by letting $\Gamma_+(\mathbf{x}^{\pm};\beta)=\Gamma_+(\beta)\prod^n_{i=1}\Gamma_+(x_i)\Gamma_+(x^{-1}_i)$. From Theorem \ref{th30} and Corollary \ref{c:orth}, we directly get the following results for partition $\lambda=(\lambda_1,\dots,\lambda_n)$.
\begin{align*}
s^{DB}_\lambda(x;\beta)&= \langle0|\Gamma_+(\mathbf{x}^{\pm};\beta)|\lambda^{o}\rangle\\
&= \sum_{\substack{k_i\geq 0\\1\leq i\leq n}}\beta^{k_1+\cdots+k_n}o_{\lambda_1-k_1,\dots,\lambda_{n}-k_{n}}(\mathbf{x}^{\pm})\\
&= \sum_{\mu\prec\lambda}\beta^{|\lambda-\mu|}o_{\mu}(\mathbf{x}^{\pm})\\
&= \det\big{(}h_{\lambda_i-i+j}(x;\beta)-h_{\lambda_i-i-j}(x;\beta)\big{)}^n_{ i,j=1},
\end{align*}
where the third equation can be easily proved similar to \cite[Proposition 2.4]{JL2022}.
\begin{remark}\label{DBr1} We note that $s^{DB}_\lambda(x;\beta)$ slightly differs from $s^{DB}_\lambda(x;\alpha)$ introduced by Bisi and Zygouras \cite{BZ2022}, while
\begin{align}
s^{DB}_\lambda(x;\alpha)=\sum_{\mu_\varepsilon\prec\lambda}\alpha^{\sum^{n-1}_{i=1}(\lambda_i-\mu_i)+\lambda_n-\varepsilon\mu_n}o_{\mu_\varepsilon}(\mathbf{x}^{\pm}).
\end{align}
\end{remark}

\section*{Appendix}
In the Appendix, we prove \eqref{e:sp14}.
\begin{proof}
For $\mu_1\geq\lambda_1$ we have that
\begin{align}
\langle \mu^{sp}|\lambda^{sp}\rangle=\delta_{\mu_1,\lambda_1}\langle 0|Y^*_{-\mu_l}\cdots Y^*_{-\mu_2}Y_{-\lambda_2}\cdots Y_{-\lambda_l}|0\rangle.
\end{align}
Now we consider the subcase when $\mu_1<\lambda_1$. It follows from \eqref{e:com1} that
$$
\langle \mu^{sp}|\lambda^{sp}\rangle=\begin{cases}
0\qquad \mu_1+i-1\neq \lambda_i ~\text{for all}~ 1\leq i\leq l\\
(-1)^{i-1}\langle 0|Y^*_{-\mu_l}\cdots Y^*_{-\mu_2}Y_{-\lambda_1-1}\cdots Y_{-\lambda_{i-1}-1}Y_{-\lambda_{i+1}} \cdots Y_{-\lambda_l}|0\rangle, ~~~~ \mu_1+i-1= \lambda_i.
\end{cases}
$$
Continuing the process, $\langle \mu^{sp}|\lambda^{sp}\rangle$ equals to 0 or $(-1)^\epsilon\langle 0|Y^*_{-\mu_l}Y_{-\lambda_1-l+1}|0\rangle$, while $\langle 0|Y^*_{-\mu_l}Y_{-\lambda_1-l+1}|0\rangle=-\langle 0|Y_{-\lambda_1-l}Y^*_{-\mu_l-1}|0\rangle=0$ since $-\mu_l-1<0$ and by \eqref{e:com1}. We therefore obtain \eqref{e:sp14} by combining the relations between $\mu_i$ and $\lambda_i$.
\end{proof}

\end{document}